\documentclass{article}

\usepackage{amsmath}
\usepackage{amsthm}
\usepackage{amssymb}
\usepackage{todonotes}
\newcounter{todocounter}

 \usepackage{titlefoot}

\newtheorem{definition}{Definition}
\newtheorem{theorem}{Theorem}
\newtheorem{lemma}{Lemma}

\begin{document}
\title{Intersection Games and Bernstein Sets}

\AtEndDocument{%
  \par
  \medskip
  \begin{tabular}{@{}l@{}}%
    \textsc{James Atchley,}
    \textsc{1155 Union Cir, Denton, TX 76205} \\
    \textit{E-mail address}: \texttt{jamesatchley@my.unt.edu, } \\
     \textsc{Lior Fishman,}
    \textsc{1155 Union Cir, Denton, TX 76205} \\
    \textit{E-mail address}: \texttt{LiorFishman@unt.edu} \\
    \textsc{Saisneha Ghatti,}
    \textsc{1155 Union Cir, Denton, TX 76205} \\
    \textit{E-mail address}: \texttt{SaisnehaGhatti@my.unt.edu} \\
   
  \end{tabular}}

\author{James Atchley, Lior Fishman, Saisneha Ghatti}

\date{}
\maketitle

\begin{abstract}
The Banach-Mazur game, Schmidt's game and McMullen's absolute winning game are three quintessential intersection games. We investigate their determinacy on the real line when the target set for either player is a Bernstein set, a non-Lebesgue measurable set whose construction depends on the axiom of choice.

\end{abstract}
\maketitle\unmarkedfntext{2023 \textit{Mathematics Subject Classifications.} 03E25, 03E60.}
\maketitle\unmarkedfntext{\textit{Key words and phrases.} Determinacy, Intersection Games.}
\maketitle\unmarkedfntext{The authors thank the University of North Texas Incubator Program - an Undergraduate}
\maketitle\unmarkedfntext{Mentored Research Program.}

\section{Introduction}
The Banach-Mazur game was first mentioned in the so-called Scottish Book. The Scottish Book is a collection of problems raised by well-known mathematicians between the world wars in a Scottish bar (hence the name) in Poland. Some of these problems are solved, some unsolved, and some will probably stay unsolvable. As the story goes, one of the mathematicians would pose a question to the rest of his friends, and solving it would result in winning a prize, ranging from a small beer to a live goose. For this paper's purposes, the Scottish Book featured one of the first infinite positional games of perfect information. Referred to now as Problem 43 in the book, the question was proposed by S. Mazur. He defined a game, thereafter known as the Banach-Mazur game, in which it is easily seen that the second player wins on a co-meager set. He conjectured that this winning condition is not only sufficient but a necessary one as well. Banach proved his conjecture when playing on the real line and in 1957, J. Oxtoby \cite{Oxtoby} proved a generalized version of Mazur's conjecture. Coincidentally, D. Mauldin, a faculty member of our home university, the University of North Texas was an editor for a modern translation of the Scottish Book \cite{Mauldin}.
In 1966, W. Schmidt \cite{Schmidt}, used a modified version of the Banach-Mazur game (thereafter known as Schmidt's game) to prove that the set of badly approximable numbers (vectors) has a full Haussdorff dimension in $\mathbb{R}^d$. More recently, in 2009, C.T. McMullen \cite{McMullen} introduced yet another version of the Banach-Mazur game to show that in this new game, winning sets are preserved under quasiconformal maps. He named his version of the game as the Absolute Winning Game, and we will call it McMullen's game. All games we consider in this paper can be played on any Polish space, i.e., complete, metrizable and separable spaces. In what follows we will consider the Real line as our playground. The generalization to other Polish spaces is obvious.
By the axiom of determinacy, these games are determined (see definition \ref{STR}), that is, one of the players always has a winning strategy.  The Axiom of Determinacy, introduced in 1962 by J. Mycielski and H. Steinhaus, states that certain games (Gale Stewart games) are determined. For what will prove essential for this paper, Mycielski and S. Swierczkowski proved that the Axiom of Determinacy implies that all sets in the real line $\mathbb{R}$ are Lebesgue measurable, a fact incompatible with the axiom of choice. 

The most well-known example of a non-Lebesgue measurable set is a Vitali set, constructed by using the axiom of choice. Another example of a set that is not Lebesgue measurable and can be constructed by applying the axiom of choice was introduced by F. Bernstein in 1908, thereafter known as a Bernstein set (see \cite{Oxtoby} for further details). While the axiom of determinacy deems all of these games determined, the focus of this paper is to prove that they are not determined on Bernstein sets whose construction heavily depends on the axiom of choice.    

\section{Description of the Games}

When we talk about games here, we will be talking about games played by two players, called Alice and Bob. Bob always goes first and Alice always goes second.  Each player takes turns choosing intervals on the Real line according to the rules of the game. After infinitely many stages a winner is decided by some rule.  For us, this rule will always be connected to some subset of \(\mathbb{R}\) called the target set.

Before we continue on to describe the specific games we will be talking about here, we must actually say what it means for a player to have a winning strategy in one of these games and what it means for a set to be determined

\begin{definition}
\label{STR}
    A \textsl{strategy} for
    either player is a rule that specifies what move they will make in every possible situation. In the case of Alice, she knows which intervals $B_0, A_0, B_1, A_1\dots B_n$ have been chosen in the previous moves, and the target set is known. From this information, her strategy must tell her
    which interval to choose for $A_n$ . Thus, a strategy for Alice is a sequence of
    closed-interval-valued functions $f_n(B_0, A_0, B_1, A_1 \dots B_n)$ outputting an intevral $A_n$ that is a valid move for Alice as her nth play of the game. 
\end{definition}
It should be noted that this is not the most general definition of a strategy, it can be defined for a much wider family of games than what we have so far described,  But this definition is all we need for our purposes.

\begin{definition}
    We say a strategy for either player is winning if it guarantees that the player wins regardless of what moves the other player makes.
\end{definition}

\begin{definition}
    We say that a game is not determined on $S \subset \mathbb{R}$  if neither player has a winning strategy when $S$ is the target set.  Otherwise we say it is determined.
\end{definition}

\subsection{The Banach-Mazur game}
The Banach-Mazur game is played by Alice and Bob, with a target set $S\subset \mathbb{R}$. Bob begins by choosing a nonempty closed interval $B_0$. Alice then chooses a nonempty closed interval $A_0 \subset B_0$. 
Bob responds by playing another interval $B_1 \subset A_0$ and the game continues in this way indefinitely. Alice wins the game if $\bigcap\limits_{i=0}^\infty A_i=\bigcap\limits_{i=0}^\infty B_i$ has non-empty intersection with $S$,  otherwise Bob wins. We note that if the target set $S$ is not dense the game is trivial as Bob can effectively win on his first move. Thus, assuming $S$ is dense, we can assume that Bob will play so as to ensure that $|B_i|\rightarrow0$ where $|I|$ denotes the length of the interval $I$. 

It is known that Alice has a winning strategy if and only if $S$ is co-meager. For a more detailed discussion of the game see \cite{Oxtoby}.


\subsection{McMullen's Absolute Winning game}
McMullen's game (the Absolute Winning game) is played by Alice and Bob and both players are given a parameter $\beta$ such that $0<\beta<\frac{1}{3}$ and a target set $S \subset \mathbb{R}.$ Bob starts by choosing a nonempty closed interval $B_0$. Alice then chooses $A_0 \subset B_0$ where the length of $A_0$ is $\beta$ times the length of $B_0.$  Bob responds by playing an interval $B_1$ of length $\beta$ times the length of  $B_0$ contained in $ B_0 \setminus A_0$.  Continuing in this way, we get that for every $n>0$, \[A_n \subset B_n\] and \[B_n \subset B_{n-1} \setminus A_{n-1}.\]  Alice wins the game if $\bigcap\limits_{n=0}^\infty B_i$ has non-empty intersection with and $S$ and otherwise Bob wins. See \cite{McMullen} for a more detailed discussion of the game.

\subsection{Schmidt's game}
Schmidt's game is played by Alice and Bob who are given two parameters $0 < \alpha, \beta <1$ and a target set $S \subset \mathbb{R}^d.$ As in the Banach-Mazur game, Bob starts by choosing a nonempty closed interval $B_0$. Alice then chooses a nonempty closed interval $A_0 \subset B_0$ with length $\alpha$ times that of $B_0$.  Bob responds by playing an interval $B_1 \subset A_0.$ with length $\beta$ times that of $B_1$. The game continues in this way indefinitely.
Alice wins the game if $\bigcap\limits_{i=0}^\infty B_i$ has non-empty intersection with $S$ otherwise Bob wins. 

If for a specific $\alpha$ and $\beta$ Alice has a winning strategy, we say that $S$ is $(\alpha,\beta)$-winning. If there exists some $\alpha$ such that $S$ is $(\alpha,\beta)$-winning for every $\beta$, we say that $S$ is $\alpha$-winning. 

One of the important consequences of a set $S$ being $\alpha$-winning is that the Hausdorff dimension of $S$ is maximal. In $\mathbb{R}^d$ this means it has dimension $d$.  We note that although at first glance Schmidt's game seems very similar to the Banach-Mazur game, it can be shown that some $\alpha$-winning sets are in fact meager, e.g., the set of badly approximable numbers.
See \cite{Schmidt} for a more detailed discussion of the game.

\section{Bernstein Sets and the determinacy of games on a Bernstein set}
 
It might seem that either Alice or Bob always has a winning strategy in these games regardless of the target set. This is not necessarily the case. If we assume the axiom of choice, we can give an example of a set for which neither player has a winning strategy in any of the above games. 
\begin{definition}
    A Bernstein set is a set that has non-empty intersection with all closed uncountable sets but contains no closed uncountable set.
\end{definition}
As mentioned in the introduction, the construction of a Bernstein set heavily relies on the axiom of choice. See \cite{Oxtoby} for details. It is clear from the definition that both a Bernstein set and it's compliment (which is also a Bernstein set) are dense.  

\begin{definition}
    A perfect set $\mathcal{P}$ is a set that is equal to the set of its accumulation points.   
\end{definition}
It is well known that a Perfect set in $\mathbb{R}$ is necessarily closed and uncountable see e.g. \cite{Rudin}

\subsection{The Banach-Mazur Game}
\begin{theorem}\label{banach}
    The Banach-Mazur Game is not determined on a Bernstein set.
\end{theorem}
\begin{proof}
The goal in this proof is to show that any winning set or losing set must either contain or be disjoint from some perfect set. Any winning strategy will generate a perfect set in either the target set or it's compliment. And therefore no Bernstein set can be winning or losing.

Let Alice and Bob play the Banach-Mazur game and suppose one of the players has a winning strategy.  
At some stage of the game, the player with the winning strategy chooses a closed, nonempty interval $I_k$ as part of the winning strategy. If the player with the winning strategy is Alice let $T$ be the target set and if it's Bob let $T$ be the complement of the target set. 

There exists two disjoint closed intervals $I_{k_0}, I_{k_1}$ contained in $I_k$ that are legal moves in the game. The player with the winning strategy will have a response to both called $I_{k_0}, I_{k_1} \subset I_k$ as part of their winning strategy. There are then four disjoint intervals $I_{k_{00}}, I_{k_{01}} \subset I_{k_0}$ and $I_{k_{10}}, I_{k_{11}} \subset I_{k_1}$ which are legal moves in response to the winning strategy.  We continue in this way and identify uniquely for any finite binary sequence $j$ the interval $I_j$ which is a possible  move for the player without a winning strategy in response to the player with the winning strategy.
\noindent 
We let $\tau$ be a binary sequence. Then $\tau$ uniquely identifies a point in $T$ by the assumption that a winning strategy is being employed. Let $\omega$ be a finite binary sequence. We say that $\omega < \tau$ if the sequences match up to $\omega$'s last digit. Define $I_{k_{\tau}} = \bigcap\limits_{\omega < \tau}I_{k_{\omega}}$ and note that for any two distinct binary sequences $\tau_1, \tau_2$, $I_{k_{\tau_1}} \neq I_{k_{\tau_2}}.$ As we used a winning strategy to construct $I_{k_{\tau}}$, it must be contained within $T$. Let A be the union of all the $I_{k_{\tau}}$ over the set of all binary sequences. By definition, $A \subset T$.

\noindent
$A$ is closed as it is the intersection of closed sets and therefore contains all of its accumulation points. Let $x \in A$. There exists a binary sequence $\tau_{*}$ such that $\{ x \} = I_{k_{\tau_{*}}}.$ Let $\omega_0$ be a finite binary sequence so that $\omega_0 < \tau_{*}.$ There exists some $x_0 \in I_{k_{\omega_0}} \cap A$ not equal to $x$.
We let $\omega_1 < \tau_{*}$ be a finite binary sequence strictly longer than $\omega_0$ and there exists $x_1 \in I_{k_{\omega_1}} \cap A$ not equal to $x$. We can continue in this way so that $x_i \in A$ for all $i$, and by construction $x_i \rightarrow x$ in $\mathbb{R}$ so we have that $x$ is an accumulation point of A. Thus, every $x \in A$ is an accumulation point, so since A is closed it is a perfect set.
Suppose that $T$ is a Bernstein set and either player has a winning strategy. Then the Bernstein set would contain a perfect set which is a contradiction. Thus, the Banach-Mazur game is not determined on a Bernstein set.
\end{proof}

\subsection{McMullen's game}
\begin{theorem}
    McMullen's game is not determined on a Bernstein set.
\end{theorem}
\begin{proof}
The proof is similar to the proof of Theorem \ref{banach}, so we'll just sketch it and the reader can fill out the missing details. Let Alice and Bob play McMullen's game and let Alice have a winning strategy. Using a proof similar to that for the Banach-Mazur game, the target set must contain a perfect set. 
Now, we consider the case where Bob has a winning strategy. Let $B_0$ be some move for Bob in the winning strategy that is not his first move and let $C$ be Bob's move before $B_0$. Alice could have played in $C \setminus B_0$ as her previous move, so there also exists $B_1 \subset C \setminus B_0$ that is part of Bob's winning strategy in response to Alice's move, $C \setminus B_0.$ We can repeat this process for $B_0$ resulting in the disjoint moves $B_{00}, B_{01} \subset B_0$ and similarly for $B_1$ resulting in the disjoint moves $B_{10}, B_{11} \subset B_1.$ 
Continuing in this way, for every finite binary sequence $\omega$ we have an interval $B_{\omega}$ that is part of Bob's winning strategy. We can proceed similarly as in the Banach-Mazur game and show that the complement of the target set must contain a perfect set. Therefore, the target set cannot be a Bernstein set when Bob has a winning strategy.  So McMullen's game is not determined on Bernstein sets.

\end{proof}

\subsection{Schmidt's Game}

\begin{theorem}\label{schmidt}
    If $\beta > 2 - \frac{1}{\alpha}, \alpha > 2 - \frac{1}{\beta}$ then Schmidt's Game is not determined on a Bernstein set.
\end{theorem}

Before we prove this theorem for Schmidt's game, we will prove a lemma that will motivate an assumption that is used later. 
\begin{lemma}
If in Schmidt's game $\beta \leq 2  - \frac{1}{\alpha}$ then for any $x\in \mathbb{R}$ Bob can win on the target set $T = \mathbb{R} \setminus \{x\}.$
\end{lemma}
\begin{proof}
Fix an $x$.  Bob begins by choosing any interval with $x$ in the center. Assume without loss of generality that this interval has length one. 
 Alice's move will then necessarily contain that point since our assumptions on $\alpha, \beta$ implies $\alpha > \frac{1}{2}$. In addition, regardless of Alice's move, Bob can respond with an interval where $x$ is the center. This is because the distance between the center of Bob's initial move and the nearest edge of Alice's move is at least $\frac{\alpha \beta}{2}$ and so the closed ball, $B(x, \frac{\alpha \beta}{2})$ is contained in it.
The game can proceed in this way indefinitely and Bob wins since the intersection of all his moves will be $\{x\}$. 
In the same way, we can show that Alice can win on any dense set, even if it's only countable, when $\alpha \leq 2 - \frac{1}{\beta}$.
\end{proof}

The point of proving this is to show that these Schmidt's Game cases where $\alpha \leq 2 - \frac{1}{\beta}$ or $\beta \leq 2  - \frac{1}{\alpha}$ are somewhat uninteresting.  This is why it is fine to exclude them from the following theorem.\\

We are now ready to prove Theorem \ref{schmidt}.

\begin{proof}
If we assume $0<\beta<\frac{1}{2}$ and Alice has a winning strategy on a Bernstein set, then this is simple. As in the previous games, for each interval Alice plays, Bob has two disjoint responses as part of his strategy. At each stage of the game, the union of these possible responses is a union of closed, disjoint intervals. As shown before, the intersection of these closed sets is a perfect set which contradicts the assumption that there is a winning strategy.
We assume $\beta \geq \frac{1}{2}$. let $I_1$ be Alice's first move as part of her winning strategy on a Bernstein set. Let $x$ be the center of $I_1$. Our goal will be to prove that Bob can eventually force the game to be played on either the left or right of $x$.  We will prove this for the right side, the left side will be equivalent.

Without loss of generality let $|I_1| = 1$. Let $I_2 \subset I_1$ be Bob's response. As $\beta \geq \frac{1}{2}$, Bob's move contains $x$. Now, there are two possibilities from this point forward, $\bigcap\limits_{n=1}^\infty I_n \subset I_2 \cap (-\infty, x], \bigcap\limits_{n=1}^\infty I_n \subset I_2 \cap (x, \infty)$ where $I_n$ is the sequence of all plays in the game. Suppose that for any of Alice's moves $I_{2n - 1},$ Bob responds with an $I_{2n}$ which shares the right endpoint of $I_{2n - 1}.$ 
We define $\ell_{n}$ as the left endpoint of an interval $I_n.$ Note that \[\lim_{n\rightarrow \infty} \{\ell_{n}\} \in \bigcap\limits_{n=1}^\infty I_n.\]
We will now show that $\bigcap\limits_{n=1}^\infty I_n \subset I_2 \cap (x, \infty)$ is the only possibility when Bob is using this right endpoint strategy. Now let Bob play some interval $I_k$ as part of this strategy. By the rules of the game, $\ell_{k+1} \geq \ell_k$ and so  
\[|\ell_{k} - \ell_{k+2}| \geq |\ell_{k+1} - \ell_{k + 2}| \geq \alpha^{k+1} \beta^k - \alpha^{k+1} \beta^{k+1}.\]
Thus, 
\[\sum_{k = 0}^\infty|\ell_{2k} - \ell_{2k+2}| \geq (\alpha - \alpha \beta)\sum_{k = 0}^\infty \alpha^k \beta^k = (1 - \beta) \frac{\alpha}{1 - \alpha\beta}.\]
We now will show that $(1 - \beta) \frac{\alpha}{1 - \alpha\beta} > \beta - \frac{1}{2}$. This will show that the sequence $\ell_n$ converges to a number greater than $x$ as $n \to \infty$.  This implies that, after finitely many steps, the players will be choosing intervals on the right of $x$.  Note that:
\begin{align*}
    \alpha &> 2 - \frac{1}{\beta} \\
    \alpha\beta &> 2\beta - 1 \\
    1 - \alpha\beta &< 2 - 2\beta \\
    \frac{1}{1 - \alpha\beta} &> \frac{1}{2 - 2\beta}.
\end{align*}
We can now use this to get 
\begin{align*}
    (1 - \beta) \frac{\alpha}{1 - \alpha\beta} &> (1 - \beta) \frac{\alpha\beta}{2 - 2\beta} \\
    &= \frac{\alpha\beta}{2} \\
    &> (2 - \frac{1}{\beta}) \frac{\beta}{2} \\
    &= \beta - \frac{1}{2}.
\end{align*}
This shows that \[|\ell_0 - \lim_{n\rightarrow \infty} \ell_{2n}| > \beta - \frac{1}{2}\] and so \[\lim_{n\rightarrow \infty} \ell_n > x.\] Thus, there exists an $I_{k_1}$ in Alice's strategy so that for all $y \in I_{k_1},$ $y > x.$
If Bob instead uses a strategy where he fixes the left endpoints of Alice's moves, then we can proceed similarly and we will get an interval $I_{k_0}$ in Alice's strategy such that for all $y \in I_{k_0},$ $y < x$  so $I_{k_0} \cap I_{k_1} = \O.$
Continuing this method, we can find some disjoint intervals $I_{k_{00}}, I_{k_{01}}, I_{k_{10}}, I_{k_{11}}$ as part of Alice's strategy. If we continue constructing these intervals, similar to our proofs for the Banach-Mazur game and McMullen's game, we can show that Alice's target set contains a perfect set. 
Therefore, any $\alpha\beta-$winning set where $\beta > 2 - \frac{1}{\alpha}$ and $\alpha > 2 - \frac{1}{\beta}$ will contain a perfect set. Thus, for such $\alpha, \beta$ a Bernstein set cannot be $\alpha, \beta-$winning.

In the case where Bob has a winning strategy everything is more or less identical.  now assume that Bob has a winning strategy.  If $\alpha < \frac{1}{2},$ for each interval Bob plays, Alice has two disjoint responses. At each stage of the game, the union of these possible responses is a union of closed, disjoint intervals. The intersection of these closed sets is a perfect set.

If $\alpha \geq \frac{1}{2},$ then Alice can use the same left endpoint and right endpoint strategies as Bob used above to eventually play two disjoint intervals, $I_{k_0}, I_{k_1},$ as in the previous case. Continuing in this way, we can construct a perfect set. Since we used Bob's winning strategy to construct this set, it must be contained in the complement of the target set. Thus, the complement of the target set contains a perfect set.

So we have that if either Alice or Bob has a winning strategy on a Bernstein set, then the either that Bernstein set contains a perfect set, or it's compliment does.  Either way we get a contradiction.

\noindent
Therefore, for any $\alpha, \beta$ where $\alpha > 2 - \frac{1}{\beta}$ and $\beta > 2 - \frac{1}{\alpha},$ neither player has a winning strategy on a Bernstein set and under those parameters Schmidt's Game is not determined.
\end{proof}

\end{document}